\documentclass[a4paper,12pt]{amsart}
\usepackage[utf8]{inputenc}
\usepackage{mathtools}

\makeatletter
\@namedef{subjclassname@2020}{\textup{2020} Mathematics Subject Classification}
\makeatother

\allowdisplaybreaks[1]

\usepackage{amsmath,amsthm,amssymb}
\usepackage{hyperref}
\usepackage{amsfonts}
\usepackage{amsmath}
\usepackage{mathrsfs}
\usepackage{tabto}
\usepackage{changepage} 
\usepackage{fullpage} 
\usepackage{graphicx}
\usepackage{enumitem} 

\newcommand{\ity}{\infty}

\newcommand{\F}{\mathfrak{F}}

\numberwithin{equation}{section}
\newtheorem{theorem}{Theorem}[section]
\newtheorem{lemma}[theorem]{Lemma}
\newtheorem{corollary}[theorem]{Corollary}
\newtheorem{proposition}[theorem]{Proposition}

\newtheorem{remark}[theorem]{Remark}
\newtheorem{example}[theorem]{Example}
\newtheorem{definition}[theorem]{Definition}

\thanks{The research work of the second  author is supported by ANRF(SERB) research grant TAR/2023/000197}

\begin{document}

\title [bungee set, escaping set and filled julia set of semigroup]{On dynamics of the bungee set and the  filled julia set of a  transcendental semigroup }

\author[M. Kumari]{Manisha Kumari}
\address{Department of Mathematics, Birla Institute of Technology Mesra
Ranchi--835 215, India}

\email{phdam10052.24@bitmesra.ac.in}

\author[D. Kumar]{Dinesh Kumar}
\address{Department of Mathematics, Birla Institute of Technology Mesra
Ranchi--835 215, India}

\email{dineshkumar@bitmesra.ac.in }



\begin{abstract}

We have introduced the notion of the bungee set and the filled Julia set of a transcendental semigroup using Fatou-Julia theory. Numerous results of the bungee set of a single transcendental entire function have been generalized to a transcendental semigroup. For a transcendental semigroup having no oscillatory wandering domain, we provide some conditions for the containment of the bungee set inside the Julia set. The filled Julia set has also been explored in the context of a transcendental semigroup, and some of its properties are discussed. We have also explored some new features of the escaping set of a transcendental semigroup. The bungee set of a conjugate semigroup and an abelian transcendental semigroup has also been investigated. 
\end{abstract}



\maketitle

 \section{Introduction}
Let $f$ be a transcendental entire function and for $n\in \mathbb{N}$, let $f^n$ indicate the $n-th$ composition of $f$ with itself. Here, we study how points in the complex plane behave under these iterations, which is referred to as the dynamics of the points. As not all points in the plane act in the same way, we classify them differently. One way of classification depends on the behavior of neighbourhoods around a point. In the complex plane, the Fatou set of $f$ is the set of points whose neighbourhood around each of these points exhibits stable and predictable behavior, and is denoted by $F(f)$. Whereas, if the neighbourhood displays chaotic behavior, the point belongs to the Julia set denoted by $J(f)$. There is another bifurcation of the complex plane based on the nature of the orbit of a point. These are the escaping set that consists of all those points whose orbit escapes to infinity and is denoted by $I(f)$, which was first introduced by Eremenko (1989)\cite{eremenko1989iteration}. The filled Julia set denoted by $ K(f)$ consists of all those points whose orbit is bounded, and the bungee set denoted by $BU(f)$ consists of all those points whose orbit is neither bounded nor escapes to infinity. 


An intuitive extension of the dynamics related to the iteration of a single complex function involves studying the dynamics of compositions of two or more such functions. This naturally leads to the study of transcendental semigroups. Hinkkanen and Martin \cite{hinkkanen1996dynamics}\
made a primary contribution in this direction. They extended the classical theory of the dynamics of rational functions of a single complex variable to a broader framework involving a semigroup of rational functions. Their investigations opened the door for further exploration into semigroups of transcendental entire functions. Also, Nicks and Sixsmith \cite{nicks2019bungee} studied the bungee set of a quasiregular map of transcendental type.

Subsequently, prominent results from the dynamics of one transcendental entire function have been extended to this setting by several researchers, including Poon \cite{poon1998fatou}, ZhiGang \cite{huang2004dynamics}, and Huang and Cheng \cite{huang2013singularities}. In general, a transcendental semigroup $H$ refers to a semigroup which is generated by a family of transcendental entire functions $\{h_1,h_2,\dots \}$, where the semigroup operation is the composition of functions.

A collection $\F$ of meromorphic functions is said to be normal in a domain $D \subset \mathbb{C}$ if every sequence in $\F$ contains a subsequence that converges locally uniformly in $D$ either to a meromorphic function or to the constant $\ity$. The Fatou set $F(H)$ of a transcendental semigroup $H$ is defined to be the largest open subset on which the family of functions in $H$ forms a normal family. The Julia set $J(H)$ is the complement of the Fatou set $F(H)$. When a semigroup is generated by a single transcendental entire function $h$, it is represented by $[h]$. In this scenario, the Fatou set and Julia set are denoted by $F(h)$ and $J(h)$, respectively, which correspond to the classical Fatou and Julia sets in the theory of iteration of a single transcendental entire function.

The dynamics associated with the semigroups are more intricate than the dynamics of an individual function. Not all classical properties are preserved in the semigroup context. For instance, the Fatou set and the Julia set might not be completely invariant, and the Julia set may not be the whole complex plane even if it contains an interior point (see \cite{hinkkanen1996dynamics}).

In this paper, we have initiated the study of the bungee set of a transcendental semigroup $H$ to be denoted by $BU(H)$. We have extended the results of the dynamics of one transcendental entire function on the bungee set to the dynamics of the bungee set of a transcendental semigroup. It is shown that $BU(H)$ is always a non-empty and completely invariant set. We have also initiated the study of the filled Julia set of a transcendental semigroup $H$ to be denoted by $K(H)$ and discussed its several properties. Additionally, the dynamics of various types of semigroups, such as conjugate semigroups, abelian transcendental semigroups, and those with no oscillatory wandering domains have also been explored.\\
 \textbf{Structure of the paper.} In section \eqref{sec 2}, we define some preliminary definitions and basic facts which are relevant to our study, including singularities and oscillatory wandering domains. In section \eqref{sec 3}, we mention some known results on the bungee set of a single transcendental entire function, which we intend to generalize to the transcendental semigroup. The proof of our main results, Theorem \eqref{Theo 4.3},\eqref{Th4.5}, and \eqref{Th 4.15} is contained in section \eqref{sec 4}. We also show that various properties satisfied by the bungee set of a transcendental entire function are also satisfied by the bungee set of a transcendental semigroup. Also, the notion of the filled Julia set is defined for a transcendental semigroup, and some of its properties are discussed. Some new features of the escaping set of a transcendental semigroup have been investigated. In particular, we get a partition of the complex plane in terms of the bungee set $BU(H)$, the filled Julia set $K(H)$, and the escaping set $I(H)$. Finally, in the last section \eqref{sec 5}, we have established  some results on the containment of the bungee set within the Julia set under certain conditions. Furthermore, we have shown that for two conjugate transcendental semigroups, their bungee sets are equal.

\section{PRELIMINARIES ON TRANSCENDENTAL SEMIGROUP}\label{sec 2}
A transcendental semigroup $H$ is a semigroup which is generated by a family of transcendental entire functions $\{h_1,h_2,\dots\}$. It is denoted by $H=[h_1,h_2,\dots]$. Here, each $h\in H$ is a transcendental entire function and $H$ is closed under taking functional composition. If the generators commute among themselves, then the semigroup is called an abelian transcendental semigroup.
 Also, when the number of generators is finite, then the semigroup is said to be finitely generated.
We now introduce some basic definitions and notations for transcendental entire functions.\\
The singularities (singular values) of an entire function play an important role in its dynamics. The dynamics to a large extent is controlled by the presence of singularities. These are the  critical values, asymptotic values, and their finite limit points. For a transcendental entire function $f$, if there exists a point $z_0\in \mathbb{C}$ such that $f(z_0)=z$ and $f'(z_0)=0$ then $z_0$ and $z$ are said to be a critical point and a critical value of $f$, respectively. Also, $\eta \in \mathbb{C}$ is an asymptotic value of a transcendental entire function $f$ if there exists a curve $\gamma\to \infty$ such that $f(z)\to \eta$ as $z\to \infty $ along the curve $\gamma$.  
Denote by $Sing(f^{-1})$ the set of critical values, asymptotic values, and their finite limit points. The $Speiser \hspace{.2cm} class \hspace{.2cm} \mathcal{S}$ consists of those entire functions $f$ for which the set $Sing(f^{-1})$ is finite. Functions in $\mathcal{S}$ are said to be of finite type.  The $Eremenko-Lyubich\hspace{.2cm} class \hspace{.2cm} \mathcal{B}$ consists of those entire functions $f$ for which the  set $Sing(f^{-1})$ is bounded. Functions in $\mathcal{B}$ are said to be of bounded type. Clearly, $\mathcal{S}\subset\mathcal{B}.$   We now define the notions of critical point, critical value, and asymptotic value, respectively \cite{kumar2015dynamics}.
\begin{definition}
    Suppose a point $u_0\in \mathbb{C}$ is a critical point of some $h\in H$, then $u_0$ is called a critical point of $H$. When a point $z_0\in \mathbb{C}$ is a critical value of some $h\in H$, then $z_0$ is called a critical value of $H$. 
\end{definition}
\begin{definition}
    Suppose a point $a\in \mathbb{C}$ is an asymptotic value of some $h\in H$, then it is said to be an asymptotic value of $H$.
\end{definition}
 We now recall the notion of an oscillatory wandering domain.
\begin{definition}
  For a transcendental entire function $f$, a domain $W\subset F(f)$ is called an oscillatory wandering domain of $f$, if  $W\subseteq BU(f)$, that is, $\{\infty, b\}$ is a limit function of $\{f^n\}$ on $W$ for some $b\in \mathbb{C}$, \cite{marti2020wandering}.
\end{definition}
\section {BACKGROUND ON THE BUNGEE SET OF A SINGLE TRANSCENDENTAL ENTIRE FUNCTION} \label{sec 3}
  In this section, we list the results related to the dynamics of a single transcendental entire function, which we are interested to generalize to a transcendental semigroup.
 \subsection{Essential features of the bungee set of a single transcendental entire function}
Suppose $f$ is a transcendental entire function. Then $BU(f)$ satisfies the following properties:\\
i) $BU(f) \neq\emptyset$,\cite{osborne2015set};\\
ii) $BU(f)\cap J(f)\neq\emptyset$, (see \cite{eremenko1989iteration});\\
iii) $BU(f)\cap I(f)=\emptyset$;\\
iv)  $BU(f)\cap K(f)=\emptyset$;\\
v) If $BU(f) \cap U\neq\emptyset$ ,where $U$ is a Fatou component then $U \subseteq BU(f)$ and $U$ becomes a wandering domain of $f$, \cite[Theorem 1.1]{osborne2015set};\\
vi) $J(f)= \partial BU(f)$, \cite[Theorem 1.1]{osborne2015set};  \\
vii) $BU(f)$ is not closed, \cite{kumar2024results};\\
viii) $BU(f)\subset J(f)$ for a transcendental entire function $f$ having no oscillatory wandering domain,  \cite{kumar2024results};
\\
ix) If $f$ and $g$ are two transcendental entire functions, then $g(z) \in BU(g\circ f)$ implies that $z \in BU(f\circ g )$, \cite[Theorem 2.3]{kumar2020dynamics};\\
x) Suppose $f$ and $g$ are commuting transcendental entire functions which have no finite asymptotic values, then $BU(f\circ g) \subseteq BU(f) \cap BU(g) \subseteq BU(f)\cup BU(g)$, \cite[Theorem 3.9]{kumar2020dynamics}.\\ We now show that $BU(f)$ is completely invariant for a transcendental entire function $f$.

\begin{theorem}

   For a transcendental entire function $f$, the bungee set is completely invariant.
\end{theorem}
\begin{proof}
    For a transcendental entire function $f$, bungee set, $BU(f)=\mathbb{C}\backslash (K(f)\cup I(f))$. Now, by the complete invariance of $K(f)$ and $I(f)$\cite{kumar2024results}, it follows that $BU(f)$ is also completely invariant. 
\end{proof}
We now give an explicit example of the bungee set of a rational map.
\begin{example} Consider the rational map 
   $R(z) = 1/z^{2}$. We show that $BU(R) = \{ z \in \mathbb{C} : |z|< 1 \cup |z|>1\}$.
   \end{example}
\begin{proof}
We first consider $|z|<1$. It can be easily seen that the orbit of any point in $|z|<1$ under $1/z^2$ contains two subsequences, such that one is bounded and the other escapes to infinity. This implies that $(|z|<1) \subseteq BU(R)$. We now consider $|z|>1$. Here again, it can be seen that the orbit of any point in $|z|>1$ under $1/z^2$ contains two subsequences, one is bounded and the other escapes to infinity. This also implies that $(|z|>1)\subseteq BU(R)$. Hence, combining both cases, we conclude that $BU(R) = \{ z \in \mathbb{C} : |z|< 1 \cup |z|>1\}$. 
    \end{proof}

\section{BUNGEE SET OF A TRANSCENDENTAL SEMIGROUP}\label{sec 4}
\
In this section, we explore the bungee set of a transcendental semigroup. As noted in the previous section, the bungee set has been studied for a single transcendental entire function. We now extend the study in the context of a transcendental semigroup. Furthermore, we examine the associated dynamical properties.
We now define the bungee set of a transcendental semigroup. 
\begin{definition} \label{Def 3.1}
    Let $H =[h_1,h_2, \dots]$ be a transcendental semigroup. The bungee set of $H$ denoted by $BU(H)$ is defined as
$BU(H) = \{z\in \mathbb{C} \hspace{.2cm}| \hspace{.2cm}$ there exist at least two subsequences say $\{h_{m_k}\}, \{h_{n_k}\} $ in $H$ and a constant $R>0$ such that $|h_{m_k}(z)|<R$ and $|h_{n_k}(z)| \to\infty \}$, where both $m_k$ and $n_k $ tends to infinity as $k\to \infty$.
\end{definition} 
 \begin{remark}\label{rem 4.1}
     
 It is easy to observe that  $\cup_{h\in H} BU(h) \subseteq BU(H) $.
 \end{remark}
 We now show that the bungee set of a transcendental semigroup is always a non-empty set.
\vspace{.2cm}

\begin{lemma}
   For a transcendental semigroup $H=[ h_1,h_2, \dots]$, $BU(H) \neq\emptyset.$
\end{lemma}
  \begin{proof}
  We already know that for a transcendental entire function $f,$ $BU(f) \neq \emptyset.$ Also, the elements of $H$ are the finite compositions of its generators. For, any $h\in H,$ we have $BU(h) \neq\emptyset$. Using remark \eqref{rem 4.1}, $ \cup_{h\in H} BU(h)\subseteq BU(H).$   
   This shows that $BU(H)\neq\emptyset$. 
  \end{proof}
  As in the case of a single transcendental entire function, the following results show that the bungee set of a transcendental semigroup always intersects with the Julia set of the semigroup.
\begin{theorem}\label{Theo 4.3}
   Consider a transcendental semigroup $H=[ h_1,h_2, \dots]$.
    Then $BU(H) \cap J(H)\neq\emptyset.$
\end{theorem}
\begin{proof}

 Suppose, on the contrary, that $BU(H)\cap J(H)=\emptyset$. Then $ BU(H)\subset F(H).$ This implies that $\cup_{h\in H}BU(h)\subset F(H)$ which is further contained in $ \cap_{h\in H}F(h) $ \cite{kumar2015dynamics}. This shows $\cup_{h\in H} BU(h)\subset F(h)$ for all $h\in H$. This gives $BU(h)\subset F(h)$ for some $h\in H$. However, for a single transcendental entire function, we know that $BU(h)\cap J(h)\neq \emptyset$ \cite[Corollary 1.4]{osborne2015set}. This leads to a contradiction and hence the result.
\end{proof}

      We now show the containment of the bungee set of a semigroup inside the union of the bungee sets of individual functions in the semigroup.
      \begin{theorem}\label{th4.5}
 For a transcendental semigroup $H = [h_1,h_2, \dots]$, $BU(H) \subseteq \cup_{h\in H}BU(h)$.
        
      \end{theorem}
      \begin{proof}
            Suppose $z_0 \notin \cup BU(h)$. Then $z_0 \notin BU(h)$ for all $h\in H.$ Hence, either $h^n(z_0) \to \infty$ as $n\to\infty$ or $h^n(z_0)$ stays bounded $\forall \, h \in H$. In other words, the orbit of the point $z_0$ either escapes to infinity or is bounded. Now, any $h\in H$ is a finite composition of the generators. Therefore, its orbit will either escape to infinity or be bounded. Thus, $z_0 \notin BU(H)$ as well. Hence, the result.\\Using the above theorem and remark \ref{rem 4.1}, it now follows that  $BU(H)= \cup_{h\in H} BU(h)$.\\
            \end{proof}
     
          

Now, we recall the definition of a wandering domain of a transcendental semigroup $H$. A Fatou component $U$ is called a wandering domain of $H$ if the set $\{U_h: h\in H\}$ is infinite, where $U_h$ is the component of $F(h)$ containing $h(U)$. Otherwise, it is called a pre-periodic component of $F(H)$,\cite[Definition 3.5]{kumar2015dynamics}.
We now propose a definition of the oscillatory wandering domain of a transcendental semigroup. 
\begin{definition}
   For a transcendental semigroup $H$, a wandering domain $U\subseteq F(H)$ is called an oscillatory wandering domain, if $U\subseteq BU(H)$ and for some $b\in \mathbb{C}$,$\{\infty,b\}$ is a limit function of every sequence $\{h_n\}$ of $H$.   
\end{definition}
The next result provides an extension of \cite[Theorem 1.1]{osborne2015set} in the context of a transcendental semigroup. This provides a connection between the bungee set and the Fatou component that intersects it.
\begin{theorem} \label{Th4.5}
         Suppose $H$ is a transcendental semigroup and $U \subset F(H)$ is a Fatou component such that $U \cap BU(H)\neq \emptyset$. Then\\ (a) $U \subset BU(H)$ and $U$ is a wandering domain of $H$;\\
       (b) $J(H) = \partial BU(H)$.
     \end{theorem}
     \begin{proof}

     As $U \,\cap BU(H) \neq \emptyset$, we have $U\cap \cup_{h\in H}BU(h) \neq \emptyset$ using Theorem \eqref{th4.5}. This implies $U\cap BU(h) \neq \emptyset$ for some $h\in H$. In addition, $U\subset F(H) \subset F(h)$ as $F(H)\subset\cap_{h\in H}F(h)$, \cite{kumar2015dynamics}. Let $U\subset \Tilde{U_h} $, where $\Tilde{U_h}$ is a Fatou component of  $h$. Then $\Tilde{U_h} \cap BU(h)\neq \emptyset$ and using \cite[Theorem 1.1]{osborne2015set}, $\Tilde{U_h}$ is a wandering domain of $h$. Also, $\Tilde{U_h}\subset BU(h)\subset \cup_{h\in H} BU(h)=BU(H).$
     As $U\subset \Tilde{U_h} $, it follows that, $U$ is also a wandering domain of $H$ and $U\subset BU(H)$.
     \vspace{.2cm}\\
 To prove part (b), we will utilize the following three lemmas, where we provide the proof of the first two lemmas.
   
     \begin{lemma}\label{lem 4.8}
        For a given transcendental semigroup $H$, the set of repelling periodic points is dense in $J(H)$.
     \end{lemma}
     \begin{proof}
        We know that $J(H)=\overline{\cup_{h\in H}J(h)}$ \cite[Theorem 4.2]{poon1998fatou}. As $H$ has countably many elements, we list them as $H=[ h_1,h_2,\dots h_n,\dots ]$. Denote by $Rep\{h\}$ the set of repelling periodic points of $h$. Also, for any transcendental entire function $h$, $Rep\{h\}$ is dense in $J(h)$, that is, $J(h)=\overline{Rep\{h\}}$ \cite[Theorem 4]{bergweiler1993iteration}. Now $J(H)=\overline{J(h_1)\cup J(h_2)\cup \dots \cup J(h_n)\cup \dots}$ \\ $= \overline{\overline{Rep\{ h_1\}}\cup \overline{Rep\{ h_2\}} \cup \dots \cup \overline{Rep\{ h_n\}} \cup \dots}$. \\Let $A=\{ Rep\{h_1\}, Rep\{h_2\},\dots ,Rep\{h_n\}, \dots\}$ be the set of repelling periodic points of $H$. We conclude that the collection of repelling periodic points of $H$ is dense in $J(H)$.
        \end{proof}
        \begin{lemma}\label{lem 4.9}
             For a transcendental semigroup $H=[h_1,h_2,\dots]$, $ int BU(H)\cap J(H)=\emptyset$. 
        \end{lemma}
        \begin{proof}
           Suppose, on the contrary, that $int BU(H)\cap J(H)\neq \emptyset$ and let $z_0\in intBU(H)\cap J(H)$. Then there exists an open neighborhood $U$ of $z_0$ such that $z_0 \in U \subset BU(H).$ Also, $z_0 \in J(H)= \overline{\cup _{h \in H}Rep\{h\}}$. Every neighborhood of $z_0$ intersects with the repelling periodic points of $H$. But $BU(H)$ does not contain any periodic point. This leads to a contradiction, and hence the result.
        \end{proof}
          \begin{lemma}\label{lem 4.7}\cite[Lemma 2.1]{osborne2015set}
         Suppose $f$ is a transcendental entire function and $E\subset \mathbb C$ contains at least three points. Also, suppose that $E$ is backwards invariant under $f,$ that \hspace{.1cm} $int E \cap J(f)= \emptyset$, and that every component of F(f) that meets $E$ is contained in $E$. Then $\partial E=J(f)$.
     \end{lemma}
     \textbf{Proof of part (b) of Theorem 4.7}
         We prove this using Lemma \eqref{lem 4.7}, with $E$ replaced by $BU(H)$. We know that the bungee set is an infinite set which is backwards invariant. Also, by using Lemma \eqref{lem 4.8}, the repelling periodic points of $H$ are dense in $J(H)$. By definition, $BU(H)$ contains no periodic points. It follows from Lemma \eqref{lem 4.9} that $ int BU(H)\cap J(H)=\emptyset$. By part (a) of Theorem \eqref{Th4.5}, every component of $F(H)$ that meets $BU(H)$ is contained in $BU(H)$. Hence, $J(H)=\partial BU(H)$.
     \end{proof}
The next result provides a condition under which the bungee set is contained in its Julia set. The proof is immediate.
\begin{lemma}
   Suppose $H=[h_1,h_2,\dots]$ is an abelian transcendental semigroup. If $\infty$ is not a limit function of any subsequence in $H$ in any component of $F(H)$, then $BU(H) \subset J(H)$.
\end{lemma}
For further investigation, we now introduce the notion of the filled Julia set of a transcendental semigroup.

\begin{definition}
    Let $H=[h_1,h_2,\dots]$ be a transcendental semigroup. The filled Julia set of $H$ denoted by $K(H)$ is defined as $K(H)=\{ z\in \mathbb{C} \hspace{.2cm}| \hspace{.2mm} $  every sequence in $H$ has a subsequence which is bounded at $z$\}.
   \end{definition}

\begin{remark}\label{rem 4.13}
    As a result of the above definition, it follows that for a point $z$ to be in $K(H),$ every sequence in $H$ must be bounded at $z$. This, in particular, establishes that $K(H) \subseteq \cap_{h\in H}K(h)$.
\end{remark}
 
     It was shown in \cite{kumar2024results} that for two commuting functions $f$ and $g$ having no finite asymptotic values, the filled Julia set of the composite function $f\circ g$ contains the intersection of the filled Julia sets of the individual functions. This motivates the next result.
           
          \begin{proposition}\label{ prop 4.12}
      Suppose $H=[h_1,h_2,\dots]$ is an abelian transcendental semigroup. Then $\cap_{h\in H} K(h) \subseteq K(H) $.
          \end{proposition}
          \begin{proof}
             As $H$ is abelian, therefore any $h\in H$ can be written as $h={h_1}^{l_1}\circ {h_2^{l_2}} \circ \dots \circ{h_n^{l_n}}$ for some choice of $h_i\in H$, and $l_i\in\mathbb{N}$,$ \hspace{.2cm}1\leq i\leq n$. Now $K(h)=K({h_1}^{l_1}\circ {h_2^{l_2}} \circ \dots \circ{h_n^{l_n}})$ and by above discussion, $K(h_1^{l_1}) \cap K(h_2^{l_2}) \cap \dots \cap K(h_n^{l_n})\subset K(h).$ Also, for any entire function $g, \hspace{.2cm}K(g^n)=K(g)$ for each $n\in \mathbb{N}$. Therefore, $K(h_1) \cap K(h_2) \cap\dots \cap K(h_n) \subseteq K(h)$. Thus, $K(H)$ contains all the finite intersections of $K(h), \hspace{.2cm} h\in H$. Hence, we conclude that $\cap_{h\in H} K(h) \subseteq K(H) $.  
          \end{proof}
             From the above proposition and remark \eqref{rem 4.13}, we conclude that the filled Julia set of a semigroup equals the intersection of the filled Julia sets of individual elements, that is, $K(H)=\cap_{h\in H} K(h)$. We illustrate this with an example. 
             \begin{example}
                 Let $f=e^{\lambda z}, \lambda\in \mathbb{C}\backslash\{0\}$ and $g=f^p+  \frac{2\pi i}{\lambda},\hspace{.2cm} p\in \mathbb{N}$. Let $H=[f,g]$. Then it can be seen that $g^n=f^{np}+ \frac{2\pi i}{\lambda}$ for $n\in \mathbb{N}$. It follows that, $K(f)=K(g)$. Also, for $l,m,n, q \in \mathbb{N}$, $f^l\circ g^m=f^{l+mp}$ and $g^n\circ f^q=f^{np+q}+\frac{2\pi i}{\lambda}$. Now, for any $h\in H$, either $h=f^s $ for some $s\in \mathbb{N}$ or $h=f^{kp}+\frac{2 \pi i}{\lambda}=g^k$ for some $k\in \mathbb{N}$. In both the cases, $K(h)=K(f)$ and hence $K(H)=K(f)$.
             \end{example}
Now, we recall the definition of the escaping set of a transcendental semigroup. For a transcendental semigroup $H=[h_1,h_2, \dots]$, its escaping set is denoted by I(H). It is defined as $I(H)=\{z\in \mathbb{C}\,|\,$ every sequence in $H$  diverges to infinity at z $\}$, \cite{kumar2016dynamics}.\\ As a result, $I(H)\subseteq \cap_{h\in H}I(h)$.  In the next result, we establish the reverse containment.
      \begin{theorem}\label{th 4.16}
          Let $H = [h_1,h_2, \dots]$ be an abelian transcendental semigroup. Then $\cap_{h\in H} I(h) \subset I(H)$.
      \end{theorem}
      \begin{proof}
          Suppose $z_0\in \cap_{h\in H} I(h)$. Then, $z_0\in I(h)$ for all $h\in H$, which further implies that $z_0 \notin BU(h)\cup K(h)$. If $z_0\notin BU(h)$ for any $h\in H$, then by using Theorem \eqref{th4.5}, $z_0 \notin \cup_{h\in H} BU(h)=BU(H)$. Also, if $z_0\notin K(h)$ for all $h\in H$ then, $z_0\notin \cap K(h)=K(H)$ by Proposition \eqref{ prop 4.12}. Combining both the situations, we obtain $z_0\notin BU(H)\cup K(H)$. Hence, $z_0 \in I(H)$.
    
      Thus, we conclude that for an abelian transcendental semigroup $H$, $I(H)= \cap_{h\in H}I(h)$
        \end{proof}
       
         The next results follow easily from the definitions of the bungee set, the escaping set, and the filled Julia set of a semigroup.
\begin{proposition} \label{3.14}
     Given a transcendental semigroup $H=[h_1,h_2, \dots]$, $BU(H)\cap I(H) = \emptyset$.
\end{proposition}
 \begin{proposition}\label{3.15}
      For a transcendental semigroup $H=[h_1,h_2, \dots]$, $BU(H)\cap K(H) = \emptyset$.
 \end{proposition}
 \begin{proposition}\label{3.16}
      Given a transcendental semigroup $H=[h_1,h_2, \dots]$, $I(H)\cap K(H) = \emptyset$.
 \end{proposition} 
 Hence, for a transcendental semigroup $H$, the above results establish that the complex plane can be expressed as the disjoint union of $BU(H), K(H)$, and $I(H)$.\\
  We now recall the notion of a completely invariant set of a transcendental semigroup $H$ \cite {kumar2016dynamics}. A set $P\subset \mathbb{C}$ is said to be forward invariant under $H$, if $h(P) \subseteq P$ for every $h\in H$ and $P$ is backwards invariant under $H$, if $h^{-1}(P)=\{z \in  \mathbb{C}:h(z) \in P\} \subseteq P$ for every $h \in H$. If $P$ is both forward and backwards invariant, then it is called completely invariant under $H$.
         
 With the help of the partition obtained by the above propositions \eqref{3.14}-\eqref{3.16}, we next discuss the invariance properties of the filled Julia set, as described in the theorem below.
     \begin{theorem}\label{Th 4.15}
         
      Suppose $H=[h_1,h_2,\dots]$ is an abelian transcendental semigroup. Then the filled Julia set $K(H)$ is completely invariant.
           \end{theorem}
           \begin{proof}
               Suppose $z_0\in K(H)$. Then there exist some $M>0$ such that $|h(z_0)| < M \hspace{.2cm} \forall \hspace{.2cm} h\in H$. For each $g\in H, \hspace{.2cm}|h \circ g(z_0)|= |g(h(z_0))|$ is bounded. Then, there exist some $R>0$ such that $|h(g(z_0))|<R $. This shows $g(z_0) \in K(H)$. Hence, $K(H)$ is forward invariant under $H$. For backwards invariance, suppose $z_0 \in h^{-1}(K(H)) $, that is, $ h(z_0) \in K(H) \hspace{.2cm} \forall \hspace{.2cm} h \in H $. Then $|g\circ h(z_0)|<K \hspace{.2cm} \forall \hspace{.2cm}g\in H$ and for some $K>0$. It follows that $  z_0\in K(H)$. Hence, $K(H)$ is completely invariant under $H$.
               \end{proof}
               
By the complete invariance of the filled Julia set and the escaping set \cite[Theorem 2.4]{kumar2014escaping} of an abelian transcendental semigroup $H$, we now obtain the complete invariance of the bungee set of an abelian transcendental semigroup. This is because $BU(H)=\mathbb{C}\setminus (I(H)\cup K(H))$.\\

\section{IMPORTANT OBSERVATIONS}\label{sec 5}
In this section, we focus on various properties of a transcendental semigroup, delivering key theorems, propositions, and results that highlight its dynamical behaviour.
The following result provides a condition under which the bungee set of a transcendental entire function $f$ is contained in its Julia set.
\begin{lemma} \label{lem 3.9}
Suppose $f$ is a transcendental entire function having no oscillatory wandering domain, then $BU(f)\subset J(f)$.
\end{lemma}   
\begin{proof}
    As $F(f)$ has no oscillatory wandering domain, therefore, $BU(f)\cap F(f)=\emptyset$, which implies the result.
\end{proof}
We extend the above result to a semigroup.
\begin{theorem} \label{Th 5.2}

Let $H =[ h_1,h_2,\dots] $ be a transcendental semigroup having no oscillatory wandering domain. Then $ BU(H) \subset J(H).$
\end{theorem}
\begin{proof}

 Assume this is not the case, that is, $BU(H)\cap F(H)\neq \emptyset$. As a result, there is a Fatou component $U$ such that $ U \cap BU(H) \neq \emptyset$. By Theorem \eqref{Th4.5}, $U\subset F(H)$ is a wandering domain and $U\subset BU(H)$. But, by hypothesis, $H$ has no oscillatory wandering domain. This leads to a contradiction and hence, the result.
\end{proof}
Observe that the closure of a completely invariant set is also completely invariant. We now use the minimality property of the Julia set to obtain the following result. Its proof is immediate.
\begin{lemma}
    Suppose $f$ is a transcendental entire function. Then, $J(f)\subset\overline{BU(f)}$. \label{Lemma 4.18}
\end{lemma}
Combining Lemma \eqref{lem 3.9} and Lemma \eqref{Lemma 4.18}, the following result is evident for a transcendental entire function $f$.
\begin{proposition} \label{pro 3.20}
    Suppose $f$ is a transcendental entire function having no oscillatory wandering domain. Then, $J(f)=\overline{BU(f)}$.\label{Prop 4.17}
\end{proposition}
  For a rational semigroup $G$, $J(G)$ is the smallest closed backwards invariant set containing at least three points \cite[Corollary 3.2]{hinkkanen1996dynamics}. We now proceed to show that a similar result also holds for a transcendental semigroup $H$. 
  

\begin{theorem} \label{Th 3.18}
     For a transcendental semigroup $H$, let $E$ be a closed, completely invariant set containing at least three points. Then $E\supset J(H)$.
\end{theorem}
\begin{proof}
    Let $U=E^c$. Then the family $H$ is normal on $U$ by Montel's Theorem. Therefore, $U\subset F(H)$ and this implies $U^c \supset J(H)$. Hence, $E \supset J(H)$.
\end{proof}

\begin{corollary}\label{cor3.20}
    On taking $E=\overline{BU(H)}$, we obtain $J(H) \subset\overline{BU(H)}$.
\end{corollary}
 On combining Theorem \eqref{Th 5.2} and Corollary \eqref{cor3.20},  the following result is now immediate for a transcendental semigroup $H$.
 \begin{proposition} \label{Pro 3.20}
     For a transcendental semigroup $H =[ h_1,h_2,\dots]$ having no oscillatory wandering domain $J(H)=\overline{BU(H)}.$
      \end{proposition}
      The following result throws light on the topological aspect of the bungee set of a semigroup.
      \begin{lemma}
       Suppose $H =[ h_1,h_2,\dots]$ is a transcendental semigroup having no oscillatory wandering domain. Then $BU(H)$ cannot be a closed subset of $\mathbb{C}$.
      
      \end{lemma}
\begin{proof}
   Suppose, on the contrary, that $BU(H)$ is closed. Then $BU(H)$ is a closed and completely invariant set containing at least three points. By using Theorem \eqref{Th 3.18}, $J(H)\subset BU(H)$. But this is not possible, as $J(H)$ contains repelling periodic points, which cannot be in the bungee set. Hence, the result.
\end{proof}

 We now construct an example of a semigroup in which the bungee set of the semigroup equals the bungee set of each element in the semigroup.
\begin{example}
    
Let $h_1=e^{\lambda z}$ and $h_2=h_1^{q}+\frac{2\pi i}{\lambda}, \hspace{.2cm }0\neq\lambda \in\mathbb{C}$, $q\in \mathbb{N}$. Let $H=[f,g]$ be a transcendental semigroup. Then, $BU(H)= BU(h)$ for any $h\in H$.

\end{example}
\begin{proof}
   It can be easily seen that $BU(h_1)=BU(h_1^n), \forall\hspace{.2cm} n\in \mathbb{N}$. If $h\in H$, then $h$ is a finite composition of the iterates of $h_1$ and $h_2$. It can be easily seen that either $h=h_1^s$ for some $s\in \mathbb N$ or $h=h_2^k$ for some $k\in \mathbb N$. Therefore, $BU(h)=BU(h_1^s)=BU(h_1)$ or $BU(h)=BU(h_2^k)=BU(h_2)$. In either of the cases, $BU(h)=BU(h_1)=BU(h_2)$. Hence, $BU(h)=BU(h_1)=BU(h_2)$ for any $h\in H$. 
    \begin{align*}
     Now, \hspace{.2cm} BU(H)
     &=\cup_{h\in H}BU(h)\\
     &=BU(h) \hspace{.2cm }for \hspace{.2cm }any \hspace{.2cm } h\in H\\
     &=BU(h_1)\\
     &=BU(h_2)\\
    \end{align*}
    hence, the result. 
\end{proof}
We now proceed to investigate a conjugate semigroup. In this direction, we first establish that for two conjugate entire functions, their bungee sets are related by the conjugating map.
\begin{proposition}\label{pro 5.10}
    Suppose $f$ and $g$ are two transcendental entire functions that are conjugate by $\phi$, that is, $\phi \circ f=g \circ \phi$ where $\phi (z)= az+b, \hspace{.2cm} 0\neq a,b \in \mathbb{C}$. Then $\phi(BU(f))=BU(g)$.
    
\end{proposition}
\begin{proof}
    Suppose $w \in \phi (BU(f))$, then there exist $z\in BU(f)$ such that $w= \phi(z)$. As $z\in BU(f)$, there exist at least two subsequences $\{m_k\}$ and $\{n_k\}$ tending to infinity such that $|f^{m_k}(z)|< R$ and $|f^{n_k}(z)|$ tends to infinity as $k\to \infty$. Now, $g^{m_k}(w)=g^{m_k}(\phi(z))= \phi(f^{m_k}(z))$. As $\{f^{m_k}(z)\}$ is bounded, so is $\{g^{m_k}(w)\}$. Also, $g^{n_k}(w)=g^{n_k}(\phi(z))=\phi(f^{n_k}(z))$. Since $\phi(z)= az+b$ and $f^{n_k}(z) \to \infty$, we have $g^{n_k}(w) \to \infty$ as $n_k \to \infty$. This shows that $\phi (BU(f)) \subseteq BU(g)$. On similar line, we obtain $BU(g) \subseteq\phi (BU(f))$. Hence, the result.
\end{proof} 
Recall \cite[Definition 4.10]{kumar2016dynamics}, two finitely generated semigroups $H$ and $H'$ are conjugate under a conformal map $\phi: \mathbb{C}\to \mathbb{C}$ if 
\\(a) They have an equal number of generators;\\
(b) Respective generators are conjugate under $\phi.$\\
Consider a finitely generated transcendental semigroup $H=[h_1,h_2,\dots,h_n ]$. Let $H'$ be the conjugate semigroup of $H,$ that is $H'=[\phi \circ h_1\circ \phi ^{-1},\dots, \phi \circ h_n\circ \phi ^{-1} ]$. If $H$ is an abelian transcendental semigroup and its generators are of bounded type, then similar behaviour holds for the semigroup $ H'$. \\
 Using Propositions \eqref{pro 3.20}, \eqref{Pro 3.20}, and \eqref{pro 5.10}, we now show that for two conjugate transcendental semigroups $H$ and $H'$, the closures of $BU(H)$ and $BU(H')$ are equivalent. 
 
 \begin{theorem} \label{Th 5.11}
     Suppose $H=[h_1,h_2,\dots,h_n ]$ is a finitely generated abelian transcendental semigroup where each generator is of bounded type. Let $H'$ be the conjugate semigroup of $H,$ that is $H'=[\phi \circ h_1\circ \phi ^{-1},\dots, \phi \circ h_n\circ \phi ^{-1} ]$, where $\phi : \mathbb{C}\to \mathbb{C}$ is the conjugating map $\phi (z)= az+b, \hspace{.2cm} 0\neq a,b \in \mathbb{C}$. If $H$ has no oscillatory wandering domain, then $\phi \overline{(BU(H)})= \overline{BU(H')}.$
  \end{theorem}
  \begin{proof}
   From \cite[Theorem 5.9]{kumar2015dynamics}, $J(H)=J(h_i) \, , 1\leq i\leq n$. Using Proposition \eqref{pro 3.20}, $J(f)=\overline{BU(f)}$ and Proposition \eqref{Pro 3.20} gives $J(H)= \overline{BU(H)}.$ On combining both the results we get $\overline{BU(H)}= \overline{BU(h_i)}, \, 1 \leq i \leq n.$ Thus 
 \begin{align*}
  \phi \overline{(BU(H))} 
  &=\phi \overline{(BU(h_i))}\\
  &= \overline{(BU(h_i'))}\\
  &= \overline{(BU(H'))}
\end{align*} 
  \end{proof}
 In  \cite[Theorem 2.1]{singh2020bungee}, it was shown that for a transcendental entire function $f$, neither a bounded periodic Fatou component of $f$ nor its boundary intersects with $BU(f)$. We establish a similar result for a transcendental semigroup $H$. 
\begin{theorem}
     For a transcendental semigroup $H$, let $V$ be a bounded periodic component of $F(H)$. Then $\partial V\cap BU(H)=\emptyset$
\end{theorem}
\begin{proof}
   As $V$ is a bounded periodic component of $F(H)$, for some constant $R>0$, we have $h(V)<R,\, \forall \hspace{.2cm} h\in H$. Suppose $\eta \in \partial V \cap BU(H)$. Then there exist a subsequence $h_{n_k}$ of $H$ such that $|h_{n_k}(\eta)|\to \infty$ as $\{n_k\}\to \infty$. For large $n_k$, say $N_k$, $|h_{N_k}(\eta)|> \epsilon_0 +R$ for some $\epsilon_0>0$. Also, $h_{N_k}$ is continuous at $\eta$, there exist $\delta>0$ such that $|z-\eta|<\delta$ implies  $|h_{N_k}(z)-h_{N_k}(\eta)|<\epsilon_0$. If $t\in (|z-\eta|< \delta)\cap V$, then $|h_{N_k}(t)-h_{N_k}(\eta)|<\epsilon_0$. Hence, $|h_{N_k}(t)-h_{N_k}(\eta)| \geq |h_{N_k}(\eta)|-|h_{N_k}(t)|>\epsilon_0$. This leads to a contradiction. Hence, the result.
\end{proof}
In \cite{kumar2020dynamics}, it was shown that for a transcendental entire function $f$, if $P$ is an unbounded completely invariant periodic component of the Fatou set, then $\partial P\cap BU(f)\neq\emptyset$.  
We establish a similar result for a transcendental semigroup $H$.
\begin{lemma}\label{lem 5.13}
    For a transcendental semigroup $H$, let $V$ be an unbounded completely invariant Fatou component, then $\partial V=J(H)$.
\end{lemma}
\begin{proof}
    Let $V$ be a completely invariant Fatou component, then its closure $\overline{V}$ is also completely invariant. By the minimality property of the Julia set, $J(H)=\overline{V}$. But $J(H)$ is disjoint from $V$, hence, $\partial V= J(H)$.
\end{proof}
\begin{theorem}
      For a transcendental semigroup $H$, let $V$ be an unbounded completely invariant periodic component of $F(H)$. Then $\partial V\cap BU(H)\neq\emptyset$.
\end{theorem}
\begin{proof}
    As $V$ is an unbounded completely invariant Fatou component, by Lemma\eqref{lem 5.13}, $\partial V= J(H)$. Also, by Theorem \eqref{Theo 4.3}, $BU(H) \cap J(H)\neq \emptyset $. Hence, $BU(H) \cap \partial V\neq \emptyset $.
\end{proof}


   
   


%

\end{document}